\documentclass[11pt]{article}
\usepackage{amsmath,amsthm,amsfonts,amssymb,amscd, amsxtra, mathrsfs, enumitem}
\usepackage{url}
\usepackage{cancel} 
\usepackage{rotating}
\usepackage[margin=2.8 cm,nohead]{geometry}
\usepackage{color}
\usepackage[pdftex]{hyperref}
\usepackage{graphicx}
\usepackage{subfig} 
\usepackage{algorithmic, enumitem}
\setlist[enumerate]{label=(\roman*), align=left}
\newtheorem{theorem}{Theorem}
\newtheorem{lemma}[theorem]{Lemma}

\newtheorem{definition}[theorem]{Definition}

\newtheorem{proposition}[theorem]{Proposition}

\newtheorem{remark}{Remark}
\newtheorem{example}{Example}

\newtheorem{algorithm}{Algorithm}

\begin{document}
\title{An Inexact Boosted Difference of Convex Algorithm for Nondifferentiable Functions}

\author{
O. P. Ferreira \thanks{Institute of Mathematics and Statistics, orizon@ufg.br, CEP 74001-970 - Goi\^ania, GO, Brazil, E-mail: {\tt orizon@ufg.br}.}
\and
B. S. Mordukhovich \thanks{ Department of Mathematics, Wayne State University, Detroit, MI 48202, USA, E-mail: {\tt aa1086@wayne.edu}  }
\and
 W. M. S. Santos \thanks{Department of Mathematics, Federal University of Piau\'{i}, Teresina, PI, Brazil, E-mail: {\tt joaocos.mat@ufpi.edu.br}}
\and
J. C. O. Souza \thanks{Piau\'{i} Institute of Technology and Department of Mathematics, Federal University of Piau\'{i}, Teresina, PI, Brazil, E-mail: {\tt joaocos.mat@ufpi.edu.br}
}}

\maketitle

\begin{abstract}
In this paper, we introduce an inexact approach to the Boosted Difference of Convex Functions Algorithm (BDCA) for solving nonconvex and nondifferentiable problems involving the difference of two convex functions (DC functions). Specifically, when the first DC component is differentiable and the second may be nondifferentiable, BDCA utilizes the solution from the subproblem of the DC Algorithm (DCA) to define a descent direction for the objective function. A monotone linesearch is then performed to find a new point that improves the objective function relative to the subproblem solution. This approach enhances the performance of DCA. However, if the first DC component is nondifferentiable, the BDCA direction may become an ascent direction, rendering the monotone linesearch ineffective. To address this, we propose an Inexact nonmonotone Boosted Difference of Convex Algorithm (InmBDCA). This algorithm incorporates two main features of inexactness: First, the subproblem therein is solved approximately allowing us for a controlled relative error tolerance in defining the linesearch direction. Second, an inexact nonmonotone linesearch scheme is used to determine the step size for the next iteration. Under suitable assumptions, we demonstrate that InmBDCA is well-defined, with any accumulation point of the sequence generated by InmBDCA being a critical point of the problem. We also provide iteration-complexity bounds for the algorithm. Numerical experiments show that InmBDCA outperforms both the nonsmooth BDCA (nmBDCA) and the monotone version of DCA in practical scenarios.

\end{abstract}


\section{Introduction}

In this paper, we address optimization problems of minimizing {\em Difference of Convex} (DC) functions, where {\em both} DC components may be {\em nondifferentiable}. Over the past few decades, DC programming has been well recognized in optimization theory and has been successfully applied across various fields including machine learning and image processing \cite{amp24}, compressed sensing \cite{YinLouQi2015}, location problems \cite{CruzNetoLopesSantosSouza2019,MN23}, biochemistry modeling \cite{amp24}, sparse optimization \cite{GotohTakeda2018}, clustering problems \cite{ARAGON2019,bmnt22}, hierarchical clustering \cite{NamGeremewReynoldsTran2018}, clusterwise linear regression \cite{BagirovUgon2018}, multicast network design \cite{GeremewNamSemenovBoginski2018}, multidimensional scaling \cite{LeTao2001,ARAGON2019}, etc.

The Difference of Convex functions Algorithm (DCA) was probably the first algorithm specifically designed for solving DC optimization problems \cite{TaoLe1997,Pham1986}. Over the years, numerous DCA variants have emerged, and their theoretical properties and practical applications have been thoroughly explored across a wide range of optimization challenges; see the surveys in \cite{DCAFirst2018,OliveiraABC}. Recent advancements in DC programming have emphasized algorithmic improvements including subgradient methods \cite{KhamaruWainwright2019}, proximal-subgradient approaches \cite{CRUZNETOetal2018,Moudafi2006,SOUZA3016,SunSampaio2003}, proximal bundle methods 
\cite{Kanzow2024,Welington2019}, double bundle methods \cite{KaisaBagirov2018}, codifferential calculus \cite{BagirovUgon2011}, inertial techniques \cite{WelingtonTcheou2019}, and coderivative-based regularized semi-Newton methods \cite{amp24,mordukhovich24}.  

Due to their simplicity, cost-effectiveness as well as reliability, DCA and its variants have become essential tools in high-dimensional, nonconvex optimization. A notable advancement in this area was the introduction and development of the Boosted DC Algorithm (BDCA) \cite{AragonEtall2022,ARAGON2017,ARAGON2019}, which accelerates the convergence of the traditional DCA by incorporating a linesearch mechanism; see also \cite{amp24,Zhang2024} among other recent publications.  This mechanism allows a larger step size compared to standard DCA, enhancing the convergence performance, and improving solution quality by avoiding poor local minima. The original BDCA was designed primarily for cases where the first DC component is differentiable, ensuring a descent direction. However, the applicability of this algorithm in its original form is limited when both components are nondifferentiable. We refer the reader to \cite{PhamDinh2024} for further discussions on open issues and recent advances in DC programming.

To address the limitations that BDCA faces when dealing with nondifferentiable DC components, the recent paper \cite{ferreira2024boosted} introduces the nonmonotone Boosted Difference of Convex Algorithm (nmBDCA). The primary distinction between nmBDCA and BDCA emerges when the first component is {\em nondifferentiable}. In such cases, the linesearch direction used in BDCA may fail to provide a descent direction rendering the standard Armijo linesearch ineffective. To overcome this, nmBDCA employs a nonmonotone Armijo-like linesearch that remains effective even when the first component lacks differentiability. This modification allows nmBDCA to handle a broader range of problems, thereby extending its applicability beyond BDCA and enhancing the efficiency of the linesearch. Furthermore, there is potential to improve BDCA's overall performance by addressing the computational challenges associated with solving the subproblem in each iteration. Given the inherent difficulties in obtaining exact solutions due to finite precision constraints, these subproblems are often solved approximately in practice.

To improve the {\em computational efficiency} of BDCA and nmBDCA, we propose in this paper an {\em Inexact nonmonotone Boosted Difference of Convex Algorithm} (InmBDCA). Our approach can be viewed as an inexact variant of the BDCA introduced in \cite{ARAGON2017} in two key aspects. First, the {\em subproblem} in each iteration is solved {\em approximately} with a relative error tolerance, and this approximate solution is used to determine the linesearch direction, which may or {\em may not be a descent direction} for the objective function. Second, an {\em inexact Armijo-type} linesearch employing a nonmonotone scheme is used to select the step size for the next iteration. To the best of our knowledge, inexact versions of BDCA or nmBDCA have not been previously explored in the literature. The motivation for this paper is to establish theoretical foundations for the computational behavior of BDCA and nmBDCA under finite precision. Observe that various problems have been effectively addressed using an inexact framework in other studies providing computational advantages; see, e.g., \cite{cruz2022abstract,kmt24a,kmt24b,liu2022inexact,nakayama2024inexact,WelingtonTcheou2019,souza2015proximal,SOUZA3016,SunSampaio2003,zhang2024inexact}. Under appropriate assumptions, we establish that InmBDCA is {\em well-defined} and that any accumulation point of the iterative sequence is a {\em critical point} of the problem. We also provide {\em iteration-complexity bounds} for the algorithm. Numerical experiments demonstrate that InmBDCA outperforms both the nonmonotone BDCA (nmBDCA) and the monotone DCA in practical scenarios.

The rest of this paper is structured as follows. Section~\ref{sec2} reviews relevant definitions and preliminary results that are used throughout the entire paper. In Section~\ref{sec3}, we formulate the DC problem under consideration and discuss the standing assumptions together with related issues. The novel inexact boosted DC algorithm is designed in Section~\ref{sec4} with establishing its well-definiteness and some important properties. Section~\ref{sec5} investigates various strategies to select an algorithm parameter related to the nonmonotone Armijo linesearch in the proposed algorithm. 
Section~\ref{sec6} is devoted to convergence analysis of the new algorithm, and the subsequent Section~\ref{sec7} establishes iteration-complexity bounds for iterates. In Section~\ref{sec8}, we present numerical illustrations of our new algorithm by using MATLAB software. The final Section~\ref{sec9} summarizes the obtained results and discussed some directions of our future research.


\section{Preliminaries from Convex Analysis}\label{sec2} 

This section presents some notions and results from convex analysis that are used throughout the paper. For simplicity, we consider real-valued functions $f:\mathbb{R}^{n}\to  \mathbb{R} $, although many definitions and statements hold for extended-real-valued functions, which may take value of (plus) $\infty$. Recall that the convexity of $f$ means that 
$f(\lambda x + (1-\lambda)y)\leq \lambda f(x) + (1-\lambda) f(y)$, for all $x,y\in {\mathbb{R}^{n}}$ and $\lambda \in [0,1]$. We say that $f$ is {\em strongly convex} with modulus $\sigma>0$ if the function
\begin{equation}\label{def:cssf}
f(x)- (\sigma/2)\|x\|^2\;\mbox{ is convex for }\;x\in\mathbb{R}^n.
\end{equation}

Recall that any convex function $f:\mathbb{R}^{n}\to \mathbb{R}$ is {\em locally Lipschitzian} on $\mathbb{R}^n$, i.e., for each $x\in\mathbb{R}^n$ there exists a neighborhood $U_x$ of $x$ and a constant $L_x\ge 0$ such that  
$$
|f(x)-f(y)|\leq L_{x}\|x-y\|;\mbox{  whenever }\;y\in U_{x}.
$$
The {\em global Lipschitz continuity} of a mapping  $F\colon\mathbb{R}^n\to\mathbb{R}^n$ means the existence of $L\ge 0$ with
$$ \|F(x)-F(y)\|\le L\|x-y\|\;\mbox{ for all }\;x,y\in\mathbb{R}^n.
$$

If $f:\mathbb{R}^n\to\mathbb{R}$ is differentiable on $\mathbb{R}^n$ and its gradient is globally Lipschitzian on $\mathbb{R}^n$ with constant $L$, then we have the estimate (see \cite[Lemma~5.7]{Amir}):
\begin{eqnarray}\label{global_est}
f\left(x+ \lambda d\right) \leq f(x) +\lambda  \left\langle \nabla  f(x),  d \right\rangle + L\lambda^2 \|d\|^2/2\;\mbox{ for all }\;x, d\in \mathbb{R}^n,\;\lambda\in \mathbb{R}.
\end{eqnarray}

The classical {\em subdifferential} (collection of subgradients) of a convex function $f$ on $\mathbb{R}^n$ at a given point $x\in\mathbb{R}^n$ is defined by 
\begin{equation}\label{sub}
\partial f(x):=\big\{y\in\mathbb{R}^n\;:\;f(z) \geq f(x) + \langle y, z - x \rangle\;\mbox{ for all }\;z \in \mathbb{R}^n\big\}. 
\end{equation}
The fundamental properties of the subgradient set in  \eqref{sub} presented below are well known in convex analysis, see, e.g., \cite[Proposition 2.47]{MN23}.

\begin{proposition}\label{nonemptiness}
Let $f:\mathbb{R}^n \to \mathbb{R}$ be a convex function. Then for every $x \in \mathbb{R}^n$, the subdifferential $\partial{f(x)}$ is a nonempty, convex, and compact set. Moreover, for any $d\in\mathbb{R}^n$, we have $f'(x;d)= \max_{y \in \partial f(x)}{\langle y, d \rangle}$, where $f'(x;d)$ stands for the classical directional derivative of $f$ in direction $d$.
\end{proposition}

Recall yet another useful and easily checkable property of the subdifferential \eqref{sub} saying that the mapping $\partial  f$ is {\em locally bounded}, i.e., the image of $\partial f(B)$ of a bounded set $B\subset \mathbb{R}^{n}$ is bounded in $\mathbb{R}^{n}$. As an immediate consequence of this property and the subdifferential definition in \eqref{sub}, we have the following important assertion.

\begin{proposition}\label{boundednessofy^k}
Let $f:\mathbb{R}^{n}\rightarrow \mathbb{R}$ be a convex function, and let $(x^{k})_{k\in\mathbb{N}}$ be such that ~$\lim _{k\rightarrow\infty}x^{k}=\bar x$. If $(y^{k})_{k\in\mathbb{N}}$ is a sequence satisfying $y^{k}\in \partial f(x^{k})$ for every $k\in \mathbb{N}$, then $(y^{k})_{k\in\mathbb{N}}$ is bounded and its accumulation points belongs to $\partial f(\bar x)$.
\end{proposition}

\begin{remark}\label{remark:gradcont}
It is well known that if $f:\mathbb{R}^n \to \mathbb{R}$ is a convex function and differentiable at $ \bar{x} \in \mathbb{R}^n$, then $\partial f(\bar{x}) = \{\nabla f(\bar{x})\}$.  Indeed, Proposition\ref{boundednessofy^k}  implies  that for any sequence  $(x^k)_{k \in \mathbb{N}}$ converging to $\bar{x}$, the sequence $(\nabla f(x^k))_{k \in \mathbb{N}}$ is bounded, and any of its cluster points equals $\nabla f(\bar{x})$. This result ensures that $\nabla f$ is continuous on $\mathbb{R}^n$.
\end{remark}

Now we return to strong convexity from \eqref{def:cssf} and formulate its subdifferential descriptions that can be found in \cite[Theorem~5.24]{Amir}.

\begin{proposition}\label{equistronglyconvex} We have the equivalent statements:
\begin{enumerate}
\item[\bf(i)] $f:\mathbb{R}^{n}\to  \mathbb{R} $ is a strongly convex function with modulus $\sigma>0$ as in \eqref{def:cssf}.
\item[\bf(ii)] $f(y)\ge f(x) + \langle v, y-x \rangle + ({\sigma}/{2}) \| y-x\Vert ^{2}$ for all $x,y\in \mathbb{R} ^{n} $ and  all $v\in \partial  f(x)$.
\item[\bf(iii)] $\langle w-v,x-y \rangle \geq \sigma \| y-x\Vert ^{2}$ for all $x,y\in \mathbb{R} ^{n}$,  all $w\in \partial  f(x)$ and  all  $v\in \partial  f(y).$
\end{enumerate}
\end{proposition}
The next useful statement is taken from \cite[Lemma~5.20]{Amir}.

\begin{proposition}\label{lema_soma} Let $f:\mathbb{R}^{n}\to  \mathbb{R} $ be a strongly convex function with modulus $\sigma>0$, and let $\bar{f}:\mathbb{R}^{n}\to  \mathbb{R}$ be convex. Then $f + \bar{f}$ is a strongly convex function with modulus $\sigma>0$.
\end{proposition}

The following properties of strongly convex functions (see 
\cite[Theorem~5.25 and Corollary~3.68]{Amir}) are important in convex optimization.

\begin{proposition}\label{str-conv}
Let $f:\mathbb{R}^{n}\to  \mathbb{R} $ be a differentiable and strongly convex function, and let $C \subset \mathbb{R}^n$ be a closed and convex set. Then $f$ admits a unique minimizer $\bar x\in C$  characterized by $\left\langle \nabla f(\bar x),x-\bar x\right\rangle \geq 0$ for all $x\in C$. 
\end{proposition}

Along with the subdifferential \eqref{sub}, a crucial role in our inexact algorithmic developments is paid by the approximate version of \eqref{sub} defined as follows. Given a convex function $f$ on $\mathbb{R}^n$ and given $\varepsilon\ge 0$, the {\em $\varepsilon$-subdifferential} of $f$ at $x$ is
\begin{equation}\label{e-sub}
\partial_\varepsilon f(x):=\big\{y\in\mathbb{R}^n\;:\;f(z) \geq f(x) + \langle y, z - x \rangle - \varepsilon\;\mbox{ for all }\;z \in\mathbb{R}^n\big\}. 
\end{equation}
We refer the reader to, e.g., \cite{hiriart2013convex,mn22} for various properties and applications of \eqref{e-sub}. It follows immediately from definition than 
$0 \in \partial_{\varepsilon}{f(x)}$ if only if $f(x) \leq \inf_{z\in\mathbb{R}^n}{f(z)} + \varepsilon$, which represents an {\em approximate optimality condition}. Similarly to Propositions~\ref{nonemptiness} and \ref{boundednessofy^k}, we have the following statements; see \cite[Theorem~4.1.1 and Proposition~4.1.1]{hiriart2013convex}.

\begin{proposition}
Let $f: \mathbb{R}^n\to \mathbb{R}$ be a convex function. Then for all $x\in \mathbb{R}^n$ and $\varepsilon \geq 0$, the set $\partial_{\varepsilon}{f(x)}$ is nonempty, convex, and compact.
\end{proposition}

\begin{proposition}
Let $f:\mathbb{R}^n \rightarrow \mathbb{R}$ be a convex function, and let $(x^k)_{k\in\mathbb{N}} \subset \mathbb{R}^n$ converging to $x$, $(\varepsilon_k)_{k\in\mathbb{N}} \subset \mathbb{R}_+$ converging to $\varepsilon$, and  $(w^k)_{k\in\mathbb{N}}$ converging to $w$. If $w^k \in \partial_{\varepsilon_k}{f(x^k)}$ for all $k \in \mathbb{N}$, then we get the inclusion $w\in \partial_{\varepsilon}{f(x)}$.
\end{proposition}

\begin{definition}
	Let $f\colon\mathbb{R}^{n}\to\mathbb{R}$ be a locally Lipschitz function.
	The Clarke's subdifferential of $f$ at $x\in \mathbb{R}^{n}$ is given by $\partial _{C} f(x):=\{v \in \mathbb{R}^{n}\: : \: f^{\circ }(x;d)\geq \langle v,d \rangle, ~
\forall d \in \mathbb{R}^{n} \},$
where $f^{\circ }(x;d)$ is the generalized directional derivative of $f$ at $x$ in the direction $d$ given by
$$ f^{\circ }(x;d):= \limsup _{ u\rightarrow x, t\downarrow 0} (f(u+td)-f(u))/t.$$
\end{definition}

Given a lower semicontinuous function $f:\mathbb{R}^n \to \mathbb{R}$. We can define the \textit{Fréchet's subdifferential} at $x \in\mathbb{R}^n$ as follows:
$$\hat{\partial}f(x) := \Big\{ v \in\mathbb{R}^n \: : \: 
\liminf_{y \neq x,y \to x}{\frac{f(y) - f(x) - \langle v, y - x\rangle}{\|y - x\|} \geq 0} \Big\}.$$

\begin{definition}
    Let $f:\mathbb{R}^n \to \mathbb{R}$ be a lower semicontinuous function. The \textit{Mordukhovich's limiting subdifferential} at $x \in\mathbb{R}^n$ is given by
    $$
    \partial_M{f(x)} := \{v \in \mathbb{R}^n \: : \: \exists~ x^k \to x, f(x^k) \to f(x), v^k \in\hat{\partial}{f(x^k)}, v^k \to v \}.
    $$
\end{definition}

\section{Problem Statement and Discussions}\label{sec3}

The DC problem considered in this paper is formulated as follows:
\begin{equation}\label{dcproblem}
\min \phi(x):=g(x)-h(x)\quad  \mbox{such that} ~ x\in \mathbb{R}^{n}.
\end{equation}
Throughout the paper, we imposed the following {\em standing assumptions} on the data of \eqref{dcproblem}:
\begin{enumerate}
\item[\hypertarget{H1}{(H1)}]
$g:\mathbb{R}^{n}\to \mathbb{R}$ and $h:\mathbb{R}^{n}\to \mathbb{R}$ are both strongly convex functions with modulus $\sigma>0$.
\item[\hypertarget{H2}{(H2)}]
The infimum of the objective function is bounded from below, i.e., $\bar\phi:=\inf_{x\in \mathbb{R}^{n}} \{ \phi(x)=g(x)-h(x) \}>-\infty$.
\end{enumerate}

Before proceeding with our analysis, let us first discuss the implications and necessity of the assumptions \hyperlink{H1}{(H1)} and \hyperlink{H2}{(H2)}.

\begin{remark}{\rm 
Assumption \hyperlink{(H1)} is {\em not restrictive}. Indeed, given two convex functions $f_1$ and $f_2$, we can add a strongly convex term $({\sigma}/{2})\| x \|^{2}$ to each to obtain the strongly convex functions $g(x)=f_1(x)+({\sigma}/{2})\| x \|^{2}$ and $h(x)=f_2(x)+({\sigma}/{2})\| x \|^{2}$, both with modulus $\sigma>0$; see Proposition~\ref{lema_soma}. This modification ensures that $\phi(x) = f_1(x) - f_2(x) = g(x) - h(x)$ for all $x \in \mathbb{R}^{n}$. Therefore, introducing a simple quadratic term preserves the DC structure of the problem while ensuring the strong convexity of the components. Note that strict convexity of the objective function guarantees the {\em uniqueness} of minimizers; see Proposition~\ref{str-conv}. On the other hand, Assumption \hyperlink{(H2)} is a standard requirement in DC programming, as it guarantees that the objective function is {\em well-defined} and {\em bounded from below}. This is crucial for ensuring the {\em existence} of a minimum and for the application of optimization algorithms; see \cite{ARAGON2017,ARAGON2019,CRUZNETOetal2018,ferreira2024boosted} for more details.}
\end{remark}

Next we provide the definition of a {\em critical point} in the context of DC programming, which plays a central role in the convergence analysis of algorithms designed to solve problem \eqref{dcproblem}.

\begin{definition}\label{crit}
A point $\bar x\in \mathbb{R}^{n}$ is called {\sc critical} for $\phi$ in \eqref{dcproblem} if the subdifferentials of $g$ and $h$ intersect at $\bar x$, i.e., $\partial g(\bar x)\cap \partial h(\bar x)\not=\varnothing$.
\end{definition}
Observe that when $g$ is differentiable, the criticality in Definition~\ref{crit} reduces to the condition $\nabla g(\bar x)\in \partial h(\bar x)$. Let us discuss relationships between criticality and various notions of {\em stationary points} in DC programming. 

\begin{remark}\label{stat} {\rm Stationarity points in unconstrained optimization are expressed in the form $0\in\partial\phi(\bar x)$.
Since the objective function $\phi$ in the DC program \eqref{dcproblem} is {\em nonconvex}, we have to rely on certain robust notions of nonconvex subdifferentials satisfying adequate calculus rules. Appropriate constructions of this type for local Lipschitzian functions are {\em Clarke's generalized gradient} $\partial_C\phi(\bar x)$ (see \cite{clarke1983optimization}) and {\em Mordukhovich's limiting subdifferential} $\partial_M\phi(\bar x)$ (see \cite{mordukhovich06}). The corresponding notions are known as {\em $C$-stationarity} and 
{\em $M$-stationarity}, respectively. The above subdifferentials are related as $\partial_C\phi(\bar x)={\rm co}\,\partial_M\phi(\bar x)$, where `co' stands for the convex hull of a set. This implies that any $M$-stationary point is $C$-stationary, but not vice versa. If $g$ in \eqref{dcproblem} is ${\cal C}^1$-smooth around $\bar x$, then $C$-stationary points reduces to the criticality $\nabla g(\bar x)\in \partial h(\bar x)$, while $M$-stationarity may be essentially {\em more selective}  than criticality; see \cite{amp24}, particularly Remark~2.6 therein. If both $g$ and $h$ are nonsmooth around $\bar x$, then $C$-stationarity of $\bar x$ ensures its criticality in \eqref{dcproblem} by 
\begin{equation}\label{c-stat}
0\in\partial_C\phi(\bar x)\subset\partial_C g(\bar x)+\partial_C(-h)(\bar x)= \partial g(\bar x)-\partial h(\bar x)
\end{equation}
due to the sum rule and  plus-minus symmetry of the generalized gradient, which agrees with the subdifferential \eqref{sub} for convex functions; see \cite[Propositions~2.2.7, 2.3.1, 2.3.3]{clarke1983optimization}. While the sum rule $\partial_M\phi(x)\subset\partial_M g(\bar x)+\partial_M(-h)(\bar x)$ holds for the limiting subdifferential by \cite[Theorem~2.33]{mordukhovich06}, we don't have the symmetric property $\partial_M(-h)(\bar x)=-\partial_M h(\bar x)$, and thus the $M$-counterpart of \eqref{c-stat} fails. On the other hand, such a plus-minus symmetry holds for {\em Mordukhovich's symmetric subdifferential} of $\phi$ at $\bar x$ defined by
\begin{equation*}
\partial^0_M\phi(\bar x):=\partial_M\phi(\bar x)\cup\big(-\partial_M(-\phi)(\bar x)\big),
\end{equation*}
which satisfies comprehensive calculus rules while being nonconvex and such that $\partial_C\phi(\bar x)={\rm co}\,\partial^0_M\phi(\bar x)$; see \cite{mordukhovich06} for more details. Thus we get  
\begin{equation*}
0\in\partial^0_M\phi(\bar x)\subset\partial g(\bar x)+\partial^0_M(-h)(\bar x)= \partial g(\bar x)-\partial h(\bar x), 
\end{equation*}
since the symmetric subdifferential reduces to \eqref{sub} for convex functions; see 
\cite[Theorem~1.93]{mordukhovich06}. This tells us the {\em symmetric stationarity} $0\in\partial^0_M\phi(\bar x)$  yields the $C$-stationarity and hence the criticality of $\bar x$ in \eqref{dcproblem} but not vice versa.}   
\end{remark}

\section{Inexact Boosted DC Algorithm}\label{sec4}

In this section, we develop an inexact variant of the nonmonotone Boosted Difference of Convex Algorithm (nmBDCA), termed the {\em Inexact nonmonotone Boosted Difference of Convex Algorithm} (InmBDCA). This variant introduces a framework that accommodates {\em computational inexactness} while preserving convergence guarantees under certain assumptions. Specifically, the inexact approach can be viewed as an extension of the exact algorithm, with the latter being recovered when inexactness vanishes. This makes the InmBDCA particularly suitable for addressing DC programming problems, where exact computations are computationally demanding.

Before presenting the details of the InmBDCA, we first revisit the structure and key principles of the nmBDCA as outlined in \cite{ferreira2024boosted}. This review provides the necessary context for understanding the modifications introduced in the inexact variant and the theoretical basis for its convergence properties. The formal statement of the nmBDCA is as follows:

\hrule
\begin{algorithm} {\bf Non-monotone boosted DC Algorithm (nmBDCA)}
\begin{footnotesize}
\begin{description}
	\item[ Step 1.] Fix $\sigma > 0$ and $\xi \in \left( 0,1\right) $. Choose any initial point $x^{0} \in \mathbb{R}^n$ and set $k:= 0$.
	\item[ Step 2.]  Choose $w^k \in \partial{h(x^k)}$ and compute $y^k$ as the solution of the following convex subproblem:
\begin{equation}\label{subproblemnmBDCA}
\underset{ x \in \mathbb{R}^n}{\min} g(x)- \langle w^k, x - 	x^k \rangle.
\end{equation}
	\item[ Step 3.]Set $d^{k}:=y^k - x^k$. \textbf{If} {$d^k = 0$}, \textbf{then} STOP and return $x^k$. \textbf{Otherwise}, choose ${\nu_k} \geq 0$ (to be specified later), any $\bar{\lambda}_k \geq 0$, and set $\lambda_k :=\bar{\lambda}_k$. \textbf{While} $\phi(y^k + {\lambda_k}d^k ) > \phi(y^k) - \rho{\lambda_k}^2\|d^k\|^2 + {\nu_k}$ DO $\lambda_k :=\xi\lambda_k$.	
	\item[ Step 4.] Set $x^{k+1}:= y^k + {\lambda_k}d^k$; set $k:=k+1$ and go to the Step~2.
\end{description}
\hrule
\end{footnotesize}
\end{algorithm}

%

If  $g$ is differentiable and $\nu_k = 0$ for all $k \in \mathbb{N}$, the nmBDCA reduces to the Boosted Difference of Convex Algorithm (BDCA) as introduced in \cite{ARAGON2019}. The primary distinction between nmBDCA and BDCA occurs when $g$ is not differentiable. In such cases, the vector $d^k$ in Step~3 may {\em not} constitute a {\em descent direction}, making the standard Armijo linesearch unsuitable. To address this, nmBDCA employs a {\em nonmonotone Armijo-like linesearch}, which remains applicable even when $g$ is not differentiable. This adaptation allows nmBDCA to handle a broader class of problems, extending beyond the scope of BDCA and enhancing the performance of the linesearch. Another potential improvement for BDCA could involve focusing on the efficiency of solving subproblem~\ref{subproblemnmBDCA}. Computationally, obtaining an exact solution to subproblem \eqref{subproblemnmBDCA} can be challenging due to finite precision limitations, so these subproblems are often solved approximately in practice. This leads to a critical question: for a differentiable function $g$, {\em does the approximate solution $y^k$ to subproblem \eqref{subproblemnmBDCA} ensure that $d^k$ remains a descent direction} for the objective function $\phi$ at $y^k$, as it does in the exact BDCA formulation? Addressing this question is essential because if the approximate solutions maintain descent properties, then BDCA could be implemented with approximate solutions. Conversely, if the approximate solutions do not guarantee this property, then the nmBDCA approach may be required, even when $g$ is differentiable. Let us explore this issue by first making the following remark.

\begin{remark} \label{re:lps}
{\em If $y^k$ is the exact solution of subproblem~\eqref{subproblemnmBDCA} in Step 2 of nmBDCA and $d^k = y^k - x^k \neq 0$, then $\nabla{g(y^k)} = w^k \in \partial{h(x^k)}$. Consequently, $0=\|\nabla{g(y^k)} - w^k \| < \sigma \|y^k - x^k\|.$}
\end{remark}

Before addressing the aforementioned question, we need to establish the following lemma, which demonstrates that an inequality similar to the one in Remark~\ref{re:lps}  holds in the neighborhood of the exact solution $y^k$ of subproblem~\eqref{subproblemnmBDCA} in Step~2 of nmBDCA.

\begin{lemma}\label{le:ndd}     
Let $g:\mathbb{R}^n \to \mathbb{R}$ be convex  and  differentiable, and let $x \in\mathbb{R}^n$ and $w \in \mathbb{R}^n$ be given. If $\hat{y} \in \mathbb{R}^n$ is such that $\hat{y} \neq x$ and $\|\nabla{g(\hat{y})} - w\| < \sigma \|\hat{y} - x\|$, then there exists $r >0$ with $\|\nabla{g({y})} - w\| < \sigma \|{y} - x\|$ for all $y \in B(\hat{y};r)$, which is equivalent the desired inequality.
\end{lemma}
\begin{proof}
Consider the function $\Phi(y) = \sigma \|y - x\| - \|\nabla g(y) - w\|$. Since $\nabla g$ is continuous by  Remark~\ref{remark:gradcont}, the function $\Phi$ is also continuous. Therefore, if $\Phi(\hat{y}) > 0$, then there exists a radius $r > 0$ such that $\Phi(y) > 0$ for all $y \in B(\hat{y}; r)$. This directly implies the desired inequality.
\qed
\end{proof}
Finally, the following result together  with Lemma~\ref{le:ndd} addresses the above question. In particular,  it establishes that within a neighborhood around the exact solution $y^k$ of \eqref{subproblemnmBDCA}, the new direction $ \hat{d}^k = \hat{y}^k - x^k $ is a descent direction for $\phi$ at $\hat{y}^k$ provided $ \hat{y}^k $ is sufficiently close to $ y^k $.

\begin{proposition}
Let \(\phi: \mathbb{R}^n \to \mathbb{R}\) be a DC function defined as \(\phi(x) = g(x) - h(x)\), where \(g, h: \mathbb{R}^n \to \mathbb{R}\) are strongly convex functions with modulus \(\sigma > 0\), and assume that \(g\) is differentiable. Given \(x \in \mathbb{R}^n\) and \(w \in \partial h(x)\), consider any \(y \in \mathbb{R}^n\) such that \(y \neq x\) and \(\|\nabla g(y) - w\| < \sigma \|y - x\|\). Then, it holds that \(\phi'(y; y - x) < 0\).
\end{proposition}
\begin{proof}
Since the functions  \(g, h\) are convex,  for any $y \in \mathbb{R}^n$ we have
\begin{equation}\label{eq:phi1}
\phi'(y; y-x) = g'(y;y-x) - h'(y;y-x).
\end{equation}
Differentiability of $g$ yields $g'(y;y-x) = \langle \nabla{g(y)}, y-x\rangle$. In addition, Proposition~\ref{nonemptiness} implies that $h'(y;y-x) = \max_{s \in \partial{h(y)}}{\langle s , y - x\rangle}$. Thus it follows from \eqref{eq:phi1}  that 
\begin{equation}\label{eq:phi2}
\phi'(y; y - x) \leq \langle \nabla g(y), y - x \rangle - \langle v, y - x \rangle = \langle \nabla g(y) - v, y - x \rangle\;\mbox{ for all }\;v \in \partial h(y).
\end{equation}
On the other hand, $\langle \nabla{g(y)} - v, y -x\rangle = \langle \nabla{g(y)}  - w, y - x\rangle + \langle w - v, y - x\rangle$. Since $w \in \partial{h(x)}$ and $\partial h$ is strongly monotone with modulus $\sigma > 0$, Proposition~\ref{equistronglyconvex} implies that 
$\langle w - v, y - x \rangle \leq - \sigma \|y -x\|^2.$  Therefore, combining this inequality with the previous equality tells us that
\begin{equation}\label{eq:phi6}
\langle \nabla{g(y)} - v, y - x \rangle \leq \|\nabla{g(y)} - w\|\cdot\|y - x\| - \sigma\|y - x\|^2.
\end{equation}
Combining further \eqref{eq:phi2} with \eqref{eq:phi6} and using the fact that $\|\nabla{g(y)} - w\|<\sigma\|y - x\|$ bring us to 
\begin{eqnarray*}
\phi'(y;y-x) &\leq& \|\nabla{g(y)} - w\|\cdot \|y - x\| - \sigma\|y - x\|^2 < 0, 
\end{eqnarray*}
which concludes the proof of the proposition.
\qed
\end{proof}

To ensure the {\em computational efficiency} of the nmBDCA, we propose the Inexact nonmonotone Boosted Difference of Convex Algorithm (InmBDCA) as detailed in Algorithm~2 below. This approach can be viewed as an inexact variant of BDCA from \cite{ARAGON2017} in {\em two key aspects}. First, subproblem \eqref{subproblemnmBDCA} is solved {\em approximately} to obtain \( y^k \), which determines the search direction with a  relative error tolerance. Second, an {\em inexact Armijo-type} linesearch employing a nonmonotone scheme is used to select the step size for the next iteration. To the best of our knowledge, such inexact versions of BDCA have not been previously investigated in the literature. The motivation for this paper is to establish theoretical foundations for the computational behavior of BDCA and nmBDCA under {\em finite precision}. The formal statement of the InmBDCA is as follows:
\hrule
\begin{algorithm} {\bf Inexact nonmonotone Boosted Difference of Convex Algorithm (InmBDCA)}
\begin{footnotesize}
\begin{description}
	\item[ Step 1.] Fix $\rho > 0$, $\theta \in  \left[ 0, \frac{\sigma}{2}\right)$, $\beta \in \left( 0,1\right)$, and a sequence $(\varepsilon_k)_{k\in \mathbb{N}} \subset \mathbb{R}_+$. Choose any initial point $x^{0} \in \mathbb{R}^n$ and set $k:= 0$.
	\item[ Step 2.]  Choose $w^k \in \partial_{\varepsilon_k}{h(x^k)}$ and compute $\left( y^k, \xi^k \right)$ a solution of the following subproblem
\begin{align}
\xi^k \in \partial{g(y^k)},&	\label{2inexactnmBDCA}\\
\|w^k - \xi^k\| \leq \theta\|y^k - x^k\|.&	\label{3inexactnmBDCA}
\end{align}
	\item[ Step 3.]Set $d^{k}:=y^k - x^k$. \textbf{If} {$d^k = 0$} \textbf{then} STOP and return $x^k$. \textbf{Otherwise}, choose ${\nu_k} \geq 0$ (to be specified later), any $\bar{\lambda}_k \geq 0$, and set $\lambda_k :=\bar{\lambda}_k$. \textbf{While} $\phi(y^k + {\lambda_k}d^k ) > \phi(y^k) - \rho{\lambda_k}^2\|d^k\|^2 + {\nu_k}$ DO $\lambda_k :=\xi\lambda_k$.	
	\item[ Step 4.] Set $x^{k+1}:= y^k + {\lambda_k}d^k$; set $k:=k+1$ and go to the Step 2.
\end{description}
\hrule
\end{footnotesize}
\end{algorithm}


Let us describe the main features of InmBDCA. First, it is important to establish that InmBDCA is {\em well-defined}. The objective function of problem  \eqref{subproblemnmBDCA} is strongly convex, which guarantees the existence of a solution $y^k$. Therefore, we have $0\in \partial{g(y^k)}- w^k$ implying that $w^k\in \partial{g(y^k)}$. Hence the exact solution $y^k$ of \eqref{subproblemnmBDCA} and $\xi^k=w^k$ satisfy \eqref{2inexactnmBDCA} and \eqref{3inexactnmBDCA}, respectively, ensuring that the subproblem in InmBDCA has a solution. In addition, the nonmonotone linesearch can be conducted due to the condition
$$
\lim_{\lambda \to 0^+} (\phi(y^k + {\lambda}d^k ) - \phi(y^k) + \rho{\lambda}^2\|d^k\|^2 - {\nu_k})=- {\nu_k}<0 
$$
and the continuity of $\phi$. This verifies that InmBDCA is well-defined.  {\it Throughout this section, $(x^k)_{k\in \mathbb{N}}$  denotes the  sequence generated by Algorithm~InmBDCA.} To allow more   flexibility of the algorithm, we consider $w^k \in \partial_{\varepsilon_k} h(x^k)$ instead of $w^k \in \partial h(x^k)$, as this does not introduce additional complications in the analysis of InmBDCA. It is noteworthy that the approximate subdifferential \eqref{e-sub} could be employed at this stage. Moreover, we address the decision to maintain the standard subdifferential in \eqref{2inexactnmBDCA} rather than using the approximate subdifferential. Although employing the approximate subdifferential in this context is feasible, we opt for the standard subdifferential to ensure a more straightforward analysis by taking into account that \eqref{2inexactnmBDCA} and \eqref{3inexactnmBDCA} already represent approximations of the solution to \eqref{subproblemnmBDCA}.

\begin{remark}
{\rm  In Algorithm~InmBDCA, setting $\theta = 0$ and $\varepsilon_k = 0$ for all $k\in \mathbb{N}$, we have that $w^k \in \partial h(x^k)$, and  that \eqref{3inexactnmBDCA} yields $w^k = \xi^k$. This tells us that the subproblem in Step~2 becomes to compute $y^k$ such that $w^k \in \partial g(y^k)$, i.e.,
$$
y^k = \arg\min_{x\in \mathbb{R}^n}g(x) - \langle w^k, x - x^k \rangle,
$$
which corresponds to solving \eqref{subproblemnmBDCA}. Therefore, Algorithm~InmBDCA is an inexact version of Algorithm~nmBDCA in the sense that:
\begin{itemize}
\item[\bf(i)] Instead of computing a subgradient of $h$ at $x^k$, in InmBDCA we compute an $\varepsilon_k$-approximate subgradient of $h$ at $x^k$.

\item[\bf(ii)] Computing $y^k$ as a solution to the convex subproblem \eqref{subproblemnmBDCA} is equivalent to computing $y^k$ with $w^k \in \partial g(y^k)$. In InmBDCA, we compute $y^k$ such that there exists $\xi^k\in \partial g(y^k)$ sufficiently close to $w^k$ as in \eqref{3inexactnmBDCA}. 
\end{itemize}}
\end{remark}

\begin{remark}
{\rm By using the Cauchy-Schwarz inequality,  we have  $\langle \xi^k - w^k, y^k - x^k \rangle \leq \|\xi^k - w^k\|\|y^k - x^k\|$, which being combined with  \eqref{3inexactnmBDCA} yields 
$$\langle \xi^k - w^k, y^k - x^k \rangle \leq \theta\|y^k-x^k\|^2 \leq \displaystyle\frac{\sigma}{2}\|y^k-x^k\|^2.
$$
This allows us to conclude that
$$
\langle \xi^k, x^k - y^k \rangle + ({\sigma}/{2})\|y^k - x^k\|^2 \geq - \langle w^k, y^k - x^k \rangle. 
$$
Since $\xi^k \in \partial{g(y^k)}$ and $g$ is strongly convex with modulus $\sigma > 0$, it follows that
$
g(x^k) \geq g(y^k) + \langle \xi^k , x^k - y^k \rangle + (\sigma/2)\|x^k - y^k\|^2,
$
which  being combined with the previous inequality yields
\begin{equation}\label{1inexactnmBDCA}
g(x^k) \geq g(y^k)-\langle w^k, y^k - 	x^k \rangle.
\end{equation}}
\end{remark}

The following proposition presents a version of the descent lemma for InmBDCA, which extends the descent lemmas for DCA and BDCA that have been extensively covered in various papers; see, e.g., \cite[Proposition 3]{ARAGON2017}.

\begin{proposition}\label{descent} If $d^k=0$, then $x^k$ is a $\varepsilon_k$-critical point of $\phi$. In addition, we have the estimate $\phi(y^k) \leq \phi(x^k) - \left({\sigma}/{2} - \theta\right)\|d^k\|^2 + \varepsilon_k$ for all $k\in\mathbb N$.
\end{proposition}
\begin{proof} 
Recall that $d^k:= y^k - x^k$ and take $w^k \in \partial_{\varepsilon_k}{h(x^k)}$. By \eqref{3inexactnmBDCA}, we have $\|w^k - \xi^k\| \leq \theta \|y^k - x^k\|$.  If $d^k=0$, then $y^k = x^k$, and thus $w^k = \xi^k \in \partial{g(y^k)} = \partial{g(x^k)} \subset \partial_{\varepsilon_k}{g(x^k)} $. Therefore, $\partial_{\varepsilon_k}{g(x^k)}\cap\partial_{\varepsilon_k}{h(x^k)} \neq \varnothing$, which justifies the first statement.

To verify the second statement,  take  $w^k \in \partial_{\varepsilon_k}{h(x^k)}$ and get
$h(y^k) \geq h(x^k) + \langle w^k, y^k - x^k \rangle - \varepsilon_k$. Since the function $g$ is strongly convex with modulus $\sigma>0$, it follows from \eqref{2inexactnmBDCA} and Proposition~\ref{equistronglyconvex} that
$g(x^k) \geq g(y^k) +\langle \xi^k, x^k - y^k \rangle +({\sigma}/{2})\|x^k - y^k\|^2$. Summing up the last two inequalities and providing elementary transformations bring us to 
\begin{equation} \label{inmBDCAeq3}
\quad\quad g(y^k)-h(y^k) \leq g(x^k) - h(x^k) + \langle \xi^k -  w^k, y^k - x^k \rangle - \displaystyle\frac{\sigma}{2}\|y^k - x^k\|^2 + \varepsilon_k.
\end{equation}
On the other hand, the Cauchy-Schwarz inequality yields
$\langle \xi^k - w^k, y^k - x^k \rangle \leq \| \xi^k - w^k\|\|y^k - x^k\|$.
Since $\|w^k - \xi^k\| \leq \theta\|y^k - x^k\|$, it follows that
\begin{equation}\label{inmBDCAeq4}
\langle \xi^k - w^k, y^k - x^k \rangle \leq \theta \|y^k - x^k\|^2.
\end{equation}
Combining \eqref{inmBDCAeq3} with \eqref{inmBDCAeq4} and using $\phi(x) = g(x) - h(x)$, for all $x \in\mathbb{R}^n$, ensures that
$$
\phi(y^k) \leq  \phi(x^k) - \left(\displaystyle\frac{\sigma}{2} - \theta\right)\|y^k - x^k\|^2 + \varepsilon_k\notag,
$$
which completes the proof of the proposition due to $d^k = y^k -x^k$.
\qed
\end{proof}

The above statements confirm that a nonmonotone linesearch can be performed in InmBDCA. In the next important result, we extend the obtained assertions by providing a {\em constructive estimate} for the interval within which the stepsize ${\lambda_k}$  can be chosen. This theorem also serves as a new proof of the well-defined nature of the linesearch in InmBDCA.

\begin{theorem} \label{pr:wdef}
Assume that $d^k \neq 0$ and ${\nu_k}>0$ for each $k\in \mathbb{N}$. Then we have 
$$
\hat{\tau}_k:= {{\nu_k}}/{(g(y^k + d^k) + g(x^k)-2g(y^k) + \varepsilon_k )}>0,
$$
$$
\phi(y^k + \lambda d^k) \leq \phi(y^k) - \rho{\lambda}^2\|d^k\|^2 + {\nu_k}\;\mbox{ for all }\;\lambda \in \left( 0, \tau_k \right],
$$
where $\tau_k:= \min \{1, \hat{\tau}_k, {\sigma}/{\rho} \}$. Consequently, Algorithm~{\rm 2} is well-defined.
\end{theorem}
\begin{proof}
Take $v\in \partial{h(y^k)}$. Since $h$ is strongly convex with modulus $\sigma>0$ and $d^k = y^k -x^k$,  it follows from Proposition~\ref{equistronglyconvex} that 
\begin{equation}\label{0inmBDCAeq1}
h(y^k +\lambda d^k) \geq h(y^k) + \lambda \langle v, d^k \rangle + \displaystyle\frac{\sigma}{2}\lambda^2\|d^k\|^2\;\mbox{ for all }\;\lambda \in \mathbb{R}.
\end{equation}
In particular, for $\lambda = -1$ we get the inequality
$$
h(x^k) = h(y^k - y^k + x^k) \geq h(y^k) + \langle -v, d^k \rangle + ({\sigma}/{2})\|d^k\|^2.
$$  On the other hand, it follows  from $w^k \in \partial_{\varepsilon_k}{h(x^k)}$ that
$h(y^k) \geq h(x^k) + \langle w^k, y^k - x^k \rangle - \varepsilon_k$, which being combined with the two previous inequality gives us $\langle v, d^k \rangle \geq  \langle w^k, y^k - x^k \rangle + ({\sigma}/{2})\|d^k\|^2 - \varepsilon_k$. Therefore, we arrive at the estimate
\begin{equation}\label{0inmBDCAeq4}
\lambda \langle v, d^k \rangle \geq  \lambda \langle w^k, y^k - x^k \rangle + \displaystyle\frac{\sigma}{2}\lambda \|d^k\|^2 - \lambda \varepsilon_k\;\mbox{ for all }\;\lambda \geq 0.
\end{equation}
Getting together  \eqref{0inmBDCAeq1} and \eqref{0inmBDCAeq4} leads us to
$$
h(y^k + \lambda d^k) \geq h(y^k) + \lambda \langle w^k, y^k - x^k \rangle + \displaystyle\frac{\sigma}{2}\lambda \|d^k\|^2 - \lambda \varepsilon_k + \displaystyle\frac{\sigma}{2}\lambda^2\|d^k\|^2\;\mbox{ for all }\;\lambda \geq 0.
$$
This yields, by taking into account that $y^k$ satisfies \eqref{1inexactnmBDCA}, the estimate
\begin{equation}\label{0inmBDCAeq6}
- \left( h(y^k + \lambda d^k) - h(y^k) \right)\leq \lambda \left(g(x^k) - g(y^k) + \varepsilon_k \right) - \displaystyle\frac{\sigma}{2}\lambda(1+\lambda)\|d^k\|^2\, 
\end{equation}
for all $\lambda \geq 0$. On the other hand, the strong convexity of $g$ with modulus $\sigma > 0$ ensures that
\begin{equation}\label{0inmBDCAeq7}
\begin{array}{ll}
g(y^k + \lambda d^k) - g(y^k) &= g\left( \lambda(y^k+d^k) + (1-\lambda)y^k \right) - g(y^k)\\
&\leq \lambda g(y^k + d^k) + (1-\lambda)g(y^k) - \displaystyle\frac{\sigma}{2}\lambda(1-\lambda)\|d^k\|^2 - g(y^k)\\
&= \lambda\left( g(y^k + d^k) - g(y^k)\right) - \displaystyle\frac{\sigma}{2}\lambda(1-\lambda)\|d^k\|^2
\end{array}
\end{equation}
for all $\lambda \in \left[0,1\right]$. Combining \eqref{0inmBDCAeq6} and \eqref{0inmBDCAeq7} with $\phi(y^k + \lambda d^k) - \phi(y^k) = \left( g(y^k + \lambda d^k) - g(y^k) \right) - \left( h(y^k + \lambda d^k) - h(y^k) \right)$, we obtain
\begin{align*}
\phi(y^k + \lambda d^k) - \phi(y^k) \leq \lambda \left(g(y^k + d^k) + g(x^k) - 2g(y^k) + \varepsilon_k \right) - \sigma\lambda\|d^k\|^2 
\end{align*}
for all $\lambda \in \left[0,1\right]$.
Moreover, it follows from Proposition~\ref{equistronglyconvex} that
$$
g(y^k + d^k) \geq g(y^k) + \langle u, d^k \rangle + \displaystyle\frac{\sigma}{2}\|d^k\|^2\quad \text{and} \quad g(x^k) \geq g(y^k) + \langle u, -d^k \rangle + \displaystyle\frac{\sigma}{2}\|d^k\|^2
$$
for all $u \in \partial{g(y^k)}$, where $d^k=y^k -x^k\ne 0$. This implies that
$$
g(y^k + d^k) + g(x^k) - 2g(y^k) \geq \sigma\|d^k\|^2 > 0.
$$ 
Since ${\nu_k} > 0$ and $\varepsilon_k\geq 0$, we get $\hat{\tau}_k:= {{\nu_k}}/{(g(y^k + d^k) + g(x^k)-2g(y^k) + \varepsilon_k)}\\>0$ as stated in the theorem. Furthermore, it follows that
$$
0 < \lambda\left( g(y^k + d^k) + g(x^k) - 2g(y^k) + \varepsilon_k \right) \leq {\nu_k}\;\mbox{ for all }\;\lambda \in \left( 0, \hat{\tau}_k \right].
$$
Observe also that for $\lambda > 0$ we have the estimate
$$
- \sigma \lambda \|d^k\|^2 \leq - \rho \lambda^2 \|d^k\|^2\;\mbox{ if only if }\;\lambda \leq  {\sigma}/{\rho}.
$$
Therefore, setting $\tau_k := \min \left\{1,\hat{\tau}_k, {\sigma}/{\rho} \right\}$ brings us to the estimate
$$
\phi(y^k + \lambda d^k) \leq \phi(y^k) - \rho\lambda^2\|d^k\|^2 + {\nu_k}\;\mbox{ for all }\;\lambda \in \left( 0, \tau_k \right].
$$
Since $\beta \in (0,1)$ in Algorithm~2, it follows that $\displaystyle\lim_{j \to \infty}{\beta^j \bar{\lambda}_k} = 0$, which tells us that there exists $j \in \mathbb{N}$ sufficiently large such that $\lambda_k := \beta^j \bar{\lambda}_k$ satisfies 
$$
\phi(y^k + \lambda_k d^k) \leq \phi(y^k) - \rho{\lambda_k}^2\|d^k\|^2 + v^k.
$$
This justifies that the linesearch in Step~3 is well-defined. Finally, setting $x^{k+1} :=y^k + \lambda_k d^k$ ensures that Step~4  of the algorithm also well-defined, which completes the proof of the theorem.
\qed
\end{proof}

The next result verifies the descent property of Algorithm~2 and provides explicit estimates in terms of the given data.

\begin{proposition}\label{descent1} Assume that $d^k\ne 0$ and $\nu_k>0$ for all $k\in\mathbb N$. Then
$$
\phi(x^{k+1}) \leq \phi(x^k) - \big({\sigma}/{2} -\theta + \rho\lambda^{2}_k\big)\|d^k\|^2  + {\nu_k} + \varepsilon_k\;\mbox{ for each }\;k \in \mathbb{N}.
$$
As a consequence, we have the estimate
\begin{equation}\label{ineqdescent1}
\left({\sigma}/{2} -\theta \right)\|d^k\|^2 \leq \phi(x^k)-\phi(x^{k+1}) + v _k + \varepsilon_k\;\mbox{ whenever }\;k \in \mathbb{N}.
\end{equation}
\end{proposition}
\begin{proof}
For each $k \in \mathbb{N}$, Proposition~\ref{descent} guarantees that $\phi(y^k) \leq \phi(x^k) - \left({\sigma}/{2} - \theta\right)\|d^k\|^2 + \varepsilon_k$. Combining this with Theorem~\ref{pr:wdef} for $\lambda=\lambda_k$ and setting $x^{k+1}:= y^k + \lambda_k d^k$ yields the first inequality of the proposition. To verify the second one, observe that    $\rho{\lambda}_{k}^{2} \geq 0$ for all $k \in \mathbb{N}$, which ensures that $\left( {\rho}/{2} -\theta \right)\|d^k\|^2 \leq \left({\sigma}/{2} -\theta + \rho\lambda^{2}_k\right)$, and thus completes the proof.
\qed
\end{proof}

\section{Strategies for Choosing the Measure of Nonmonotonicity}\label{sec5}

In this section, we introduce strategies for selecting the terms $({\nu_k})_{k\in\mathbb{N}}$ in the Armijo stepsize of Algorithm~2 at Step~3. This term can be viewed as a {\em  measure of nonmonotonicity}, which allows us to maintain descent iterations in our Inexact BDCA version with both nondifferentiable functions in the DC decomposition. In \cite{ferreira2024boosted}, various strategies were extensively discussed that can be either directly applied to our inexact context, or adapted as needed. We will recall  these strategies with the necessary adjustments in our setting.

\begin{itemize}
\item[\textbf{(A1)}] Given $\Delta_{min} \in \left[ 0,1\right)$, choose any $\nu_0\ge 0$ and define the sequence $({\nu_k})_{k\in\mathbb{N}}$ as follows: for each $\Delta_{k+1} \in \left[\Delta_{min}, 1\right]$, the iterate $\nu_{k+1}$ satisfies the condition:
\begin{equation}\label{S1}
0 \leq \nu_{k+1} \leq \left(1 - \Delta_{k+1}\right)\left(\phi(x^{k}) - \phi(x^{k+1}) + {\nu_k} + \varepsilon_k\right)\;\mbox{ whenever }\;k\in \mathbb{N}.
\end{equation}
\item[\textbf{(A2)}] The sequence $({\nu_k})_{k\in\mathbb{N}}$ is such that $\sum_{k=0}^{\infty} {\nu_k} < \infty$.
\item[\textbf{(A3)}] Let $({\nu_k})_{k\in\mathbb{N}}$ be such that for every $\delta > 0$, there exists $k_0 \in \mathbb{N}$ with ${\nu_k}\leq\delta\|d^k\|^2$ for all $k\geq k_0$.
\end{itemize}

Let us discuss the listed strategies.

\begin{remark}\label{re:equi}
{\rm Since $\theta \in\left[ 0, {\sigma}/{2}\right)$ in Algorithm~2, it follows from Proposition~\eqref{descent1} that $0 \leq ({\sigma}/{2}-\theta)\|d^k\|^2 \leq \phi(x^k) - \phi(x^{k+1}) + {\nu_k} +\varepsilon_k$ for all $k \in \mathbb{N}$, and thus we can take $\nu_{k+1} \geq 0$ satisfying \eqref{S1}. In particular, when $\varepsilon_k=0$ for all $k \in \mathbb{N}$, condition \eqref{S1} recovers the strategy (S1) from \cite{ferreira2024boosted}. Moreover, if  $({\nu_k})_{k\in\mathbb{N}}$ satisfies \textbf{(A1)} with $\Delta_{min}>0$ and $\sum_{k=0}^{\infty}\varepsilon_k< \infty$, then  $({\nu_k})_{k\in\mathbb{N}}$ satisfies \textbf{(A2)}. Indeed, it follows from \eqref{S1} that
\begin{equation}\label{desDelta}
0\leq \Delta_{k+1}(\phi(x^k) - \phi(x^{k+1}) + {\nu_k} + \varepsilon_k)\leq (\phi(x^k) + {\nu_k} ) - (\phi(x^{k+1}) + \nu_{k+1}) + \varepsilon_k, 
\end{equation}
for all $k \in \mathbb{N}$.
Since $0 < \Delta_{min} \leq \Delta_{k+1}$, $\phi(x^k) - \phi(x^{k+1}) + {\nu_k} + \varepsilon_k \geq 0$, and $\phi$ satisfies \textbf{(H2)}, we have that 
$$
\Delta_{min}\sum_{k =0}^{N}\left( \phi(x^k) - \phi(x^{k+1}) + {\nu_k} + \varepsilon_k \right) \leq \phi(x^0) + \nu_0 - \bar{\phi}+ \sum_{k =0}^{N}\varepsilon_k.
$$ 
Using now $\sum_{k=0}^{\infty}\varepsilon_k <\infty$, $\Delta_{min}>0$, and 0 $\leq \nu_{k+1} \leq (1-\Delta_{min})( \phi(x^k) - \phi(x^{k+1})+ {\nu_k} + \varepsilon_k)$ gives us $\sum_{k =0}^{\infty}{\nu_k} < \infty$ and thus verifies that the sequence $\left\{ {\nu_k} \right\}_{k\in \mathbb{N}}$ satisfies \textbf{(A2)}.}
\end{remark}

The following four examples illustrate how the above strategies are realized in the settings of Algorithm~2 under appropriate conditions. 

\begin{example}
{\em Take any $\nu_0 > 0$, and define $\Delta_{k+1}$ and $v_{k}$ as follows:
\begin{eqnarray}\label{desv_k}
0<\Delta_{min} \leq \Delta_{k+1} < 1,\;\;\; 0 < {\nu_{k+1}}:= (1-\Delta_{k+1})\left({\sigma}/{2} - \theta + \rho{\lambda}_{k}^2\right)\|d^k\|^2\;\;
\end{eqnarray}
for all $k\in\mathbb{N}$.
Then Proposition \ref{descent1} tells us that $\left({\sigma}/{2} -\theta + \rho{\lambda}_{k}^2\right)\|d^k\|^2 \leq \phi(x^k) - \phi(x^{k+1}) + {\nu_k} + \varepsilon_k$. Therefore, whenever $d^k \neq 0$, we have
$$
0 < \nu_{k+1} \leq (1-\Delta_{k+1})\left(\phi(x^k) - \phi(x^{k+1}) + {\nu_k} + \varepsilon_k  \right),
$$
which ensures that $(\nu_k)_{k\in \mathbb{N}}$ defined in \eqref{desv_k} satisfies \textbf{(A1)}. Having $\Delta_{min} > 0$ and $\sum_{k=0}^{\infty}\varepsilon_k <\infty$, we deduce from Remark~\ref{re:equi} that $(\nu_k)_{k\in \mathbb{N}}$ also satisfies \textbf{(A2)}.}
\end{example}

\begin{example} {\rm Let $(C_k)_{k\in\mathbb{N}}$ be the sequence of ``cost updates" employed in the nonmonotone linesearch from \cite{zhang2004nonmonotone}, which is constructed as follows: take $0\leq \eta_{min}\leq \eta_{max} < 1$, $C_0 > \phi(x^0)$, and $Q_0 = 1$. Then choose $\eta_k \in [\eta_{min}, \eta_{max}]$ and set
\begin{equation*}
Q_{k+1}:=\eta_kQ_{k}+1, \qquad C_{k+1} := ({\eta_k}Q_kC_k + \phi(x^{k+1}))/Q_{k+1}\;\mbox{ for all }\;k\in \mathbb{N}.
\end{equation*}
Define $(\nu_k)_{k\in\mathbb{N}}$ by $\nu_k: = C_k - \phi(x^k)$ and $\Delta_{k+1}:= 1/Q_{k+1}$ for all $k\in\mathbb{N}$. This gives us 
$$
_{k+1} = (1 - \Delta_{k+1})(\phi(x^k) - \phi(x^{k+1}) + \nu_k) \leq (1 - \Delta_{k+1})(\phi(x^k) - \phi(x^{k+1}) + \nu_k + \varepsilon_k),
$$
with $k \in \mathbb{N}$. After some algebraic transformations similar to \cite{ferreira2024boosted}, we get $\Delta_{min} = (1-\eta_{max}) > 0$. Thus $(\nu_{k})_{k\in\mathbb{N}}$ satisfies \textbf{(A1)} with $\Delta_{min}>0$, and hence it satisfies \textbf{(A2)} by Remark \ref{re:equi}.}
\end{example}

\begin{example} {\rm Consider the nonmonotone scheme in \cite{grippo1986nonmonotone} and define $(\nu_k)_{k\in\mathbb{N}}$ as follows: pick an integer $M >0$, set $m_0=0$ with $0 \leq m_k \leq \min\lbrace m_{k-1} + 1, M \rbrace$ for all $k\in\mathbb N$. Setting $\phi(x^{l(k)}): = \max_{0\leq j \leq m_k}\phi(x^{k-j})$
\begin{equation*}\
\nu_k:= \phi(x^{l(k)}) - \phi(x^k)\quad \text{and}\quad 0 = \Delta_{min} \leq \Delta_{k+1} \leq \frac{\phi(x^{l(k)}) - \phi(x^{l(k + 1)})}{\phi(x^{l(k)}) - \phi(x^{k+1})}
\end{equation*}
for all $k \in \mathbb{N}$, we have that these constructions satisfy \textbf{(A1)} with $\Delta_{min}=0$.} 
\end{example}

The last example is taken from \cite[Example~4]{ferreira2024boosted}.

\begin{example} {\rm 
Let $\omega > 0$ be a constant. Then the sequence $(\nu_k)_{k\in \mathbb{N}}$ defined by ${\nu_k}:=  {\omega}\|d^k\|^2/{(k+1)}$, $k \in \mathbb{N}$, satisfies \textbf{(A3)}. Indeed, by $\lim_{k\to\infty}{\frac{\omega}{k+1}} = 0$, for any fixed $\delta > 0$ there exists $k_0\in \mathbb{N}$ with ${\omega}/{(k+1)} \leq \delta$ as $k \geq k_0$, which implies that ${\nu_k} \leq \delta\|d^k\|^2$ for such $k$.
More generally, let $\left\{u_k\right\}_{k\in\mathbb{N}}$ be a sequence of positive numbers such that $\lim_{k\to\infty}u_k =\infty$, then ${\nu_k}:= \frac{\omega}{u_k}\|d^k\|^2$ satisfies \textbf{(A3)}. To see this, fix $\delta>0$ and find $k_0 \in \mathbb{N}$ with $ 0 < \frac{\omega}{\delta} \leq {u^k}$ whenever $k\geq k_0$, i.e., ${\omega}/{u^k}\leq \delta$, which implies that ${\nu_k} = {\omega}\|d^k\|^2/{u_k} \leq \delta\|d^k\|^2$ 
for all $k \geq k_0$.}
\end{example}

\section{Convergence Analysis}\label{sec6}

The goal of this section is to present convergence results of InmBDCA in Algorithm~2. We begin with providing a sufficient condition to ensure that each {\em accumulation point} of the sequence $(x^k)_{k\in \mathbb{N}}$ is a {\em critical point} of the DC program \eqref{dcproblem}.

\begin{proposition}\label{subseqconv}
If $\lim_{k\to\infty}\|d^k\|=0$ and $\lim_{k\to\infty}\varepsilon_k=0$ in Algorithm~{2}, then each accumulation point of iterative sequence $(x^k)_{k\in \mathbb{N}}$ is critical.
\end{proposition}
\begin{proof}
Let ${\bar x}$ be a accumulation point of $(x^k)_{k\in \mathbb{N}}$, and let $(x^{k_j})_{j\in \mathbb{N}}$  be a subsequence of $(x^{k})_{k\in \mathbb{N}}$ such that $\lim_{j\to\infty}x^{k_j} = {\bar x}$. Remembering that $d^k = y^k - x^k$ and  $\lim_{k\to\infty}\|d^k\|=0$, and taking into account that  $\|y^{k_j} - {\bar x}\| \leq \|y^{k_j} - x^{k_j}\| + \|x^{k_j} - {\bar x}\|$, we have that $\lim_{j\to\infty}{\|y^{k_j} - {\bar x}\|}=0$, i.e., $\lim_{j\to\infty}y^{k_j} = {\bar x}$. Since $\xi^{k_j} \in \partial{g(y^{k_j})}$ for all $j \in \mathbb{N}$, Proposition~\ref{boundednessofy^k} implies that  the sequence $(\xi^{k_j})_{j\in \mathbb{N}}$ is bounded. Without loss of generality, suppose that this sequence is convergent and set ${\bar \xi}: = \lim_{j\to\infty}\xi^{k_j}$. It follows from \eqref{3inexactnmBDCA}  that $\|\xi^{k_j} - w^{k_j}\|\leq \theta \|y^{k_j}-x^{k_j}\|$, which  implies that  $\lim_{j\to\infty}{\|\xi^{k_j} - w^{k_j}\|}=0$.  Using $\|w^{k_j} - {\bar \xi}\|\leq \|w^{k_j} - \xi^{k_j}\| + \|\xi^{k_j} - {\bar \xi}\|$, we have that $\lim_{j\to\infty}{\|w^{k_j} - {\bar \xi}\|}=0$, i.e., $\lim_{j\to\infty}w^{k_j} = {\bar \xi}$. Pick further an arbitrary $y \in \mathbb{R}^n$ and deduce from $w^{k_j} \in \partial_{\varepsilon_{k_j}}{h(x^{k_j})}$ that $h(y) \geq h(x^{k_j}) + \langle w^{k_j}, y - x^{k_j} \rangle - \varepsilon_{k_j}$. Passing to the limit as $j \to\infty$ and employing the continuity of $h$ give us $h(y) \geq h({\bar x}) + \langle {\bar \xi}, y - {\bar x} \rangle$.  Since $y$ was chosen arbitrarily, this tells us  that ${\bar \xi} \in \partial{h({\bar x})}$. On the other hand, it follows from Proposition~\ref{boundednessofy^k} by $\lim_{j\to\infty}y^{k_j} = {\bar x}$,  ${\bar \xi} = \lim_{j\to\infty}\xi^{k_j}$, and $\xi^{k_j} \in \partial{g(y^{k_j})}$ that ${\bar \xi} \in \partial{g({\bar x})}$, which verifies that ${\bar x}$ is a critical point.
\end{proof}

The subsequent results, employing Proposition~\ref{subseqconv}, establish efficient conditions on parameters of the DC program \eqref{dcproblem} and Algorithm~2 ensuring that cluster points of the iterates $(x^{k_j})_{j\in \mathbb{N}}$  are critical points of \eqref{dcproblem} under each choice of the strategies discussed in Section~5.

\begin{theorem}\label{teoS2}
Suppose that $({\nu_k})_{k\in \mathbb{N}}$ satisfies {\rm\textbf{(A2)}}, and that $(\varepsilon_k)_{k\in \mathbb{N}}$ satisfies $\sum_{k=0}^{\infty}{\varepsilon_k} <\infty$. Then each accumulation point of  $(x_k)_{k\in \mathbb{N}}$ is critical for \eqref{dcproblem}.
\end{theorem}
\begin{proof}
Proposition~\ref{descent1} tells us that $\phi(x^{k+1}) \leq \phi(x^k) - \left({\sigma}/{2} -\theta + \rho{\lambda}_{k}^2\right)\|d^k\|^2  + {\nu_k} + \varepsilon_k$ for all $k \in \mathbb{N}$. It follows from $\theta \in \left[0, {\sigma}/{2}\right)$ that ${\sigma}/{2} -\theta > 0$. Therefore,
$$
\left(\displaystyle\frac{\sigma}{2} -\theta\right)\|d^k\|^2 \leq  \left( \displaystyle\frac{\sigma}{2} -\theta + \rho{\lambda_k}^2 \right)\|d^k\|^2\leq \phi(x^k) - \phi(x^{k+1}) + v^k + \varepsilon_k.
$$	
Taking the $N$-th partial sum in the above inequality, we get
$$
\sum_{k=0}^{N}{\left( \displaystyle\frac{\sigma}{2} -\theta \right)\|d^k\|^k} \leq \phi(x^0) - \Bar\phi + \sum_{k=0}^{N}{{\nu_k}} + \sum_{k=0}^{N}{\varepsilon_k},
$$
where $\bar{\phi}:=\inf_{x\in\mathbb{R}^n}{\phi(x)}> -\infty$. Passing to the limit in the latter inequality as $N\to\infty$ and employing \textbf{(A2)} together with $\sum_{k=0}^{\infty}{\varepsilon_k} <\infty$ yield
\begin{equation*}
\sum_{k=0}^{\infty}\left(\displaystyle\frac{\sigma}{2} -\theta\right)\|d^k\|^2 \leq \phi(x^0) - \Bar\phi + \sum_{k=0}^{\infty}{{\nu_k}} + \sum_{k=0}^{\infty}{\varepsilon_k} <\infty,
\end{equation*}
which proves that $\sum_{k=0}^{\infty}\|d^k\|^2 <\infty$. This readily ensures that $\displaystyle\lim_{k\to\infty}\|d^k\|^2=0$, and this the claimed result follows from Proposition~\ref{subseqconv}.
\end{proof}

\begin{theorem}\label{teoS3}
Suppose that $({\nu_k})_{k\in \mathbb{N}}$ satisfies {\rm\textbf{(A3)}}, and that $(\varepsilon_k)_{k\in \mathbb{N}}$ satisfies $\sum_{k=0}^{\infty}{\varepsilon_k} <\infty$. Then each accumulation point of  $(x_k)_{k\in \mathbb{N}}$ in Algorithm~{\rm 2} is critical for \eqref{dcproblem}.
\end{theorem}

\begin{proof}
Setting $\delta:= ({1}/{2})\left({\sigma}/{2} - \theta \right) > 0$ and using ${\nu_k} \leq \delta\|d^k\|^2 = 2\delta\|d^k\|^2 - \delta\|d^k\|^2$ for all $k \geq k_0$, 
we get $\delta\|d^k\|^2 \leq 2\delta\|d^k\|^2 -	{\nu_k}$. Therefore, Proposition \ref{descent1} tells us that
\begin{equation}\label{EQteoS3}
\delta\|d^k\|^2 \leq \left({\sigma}/{2} - \theta\right)\|d^k\|^2 - {\nu_k}\leq \phi(x^k) - \phi(x^{k+1}) + \varepsilon_k\;\mbox{ whenever }\;k\geq k_0.
\end{equation}
Taking the partial sum in the expression above with the notation $\displaystyle\bar{\phi}= \inf_{x\in\mathbb{R}^n}\phi(x)$ yields
$$
\delta\sum_{k=0}^{N}\|d^k\|^2 \leq \phi(x^0) - \phi(x^{N+1}) + \sum_{k=0}^{N}\varepsilon_k \leq \phi(x^0) - \bar{\phi}+ \sum_{k=0}^{N}\varepsilon_k.
$$
Passing now to the limit as $N\to\infty$ in the latter inequality  and remembering that $\sum_{k=0}^{\infty}{\varepsilon_k} <\infty$ ensures that $\sum_{k=0}^{\infty}\|d^k\|^2 <\infty$, and hence $\displaystyle\lim_{k\to\infty}\|d^k\|^2=0$. In this way, we arrive at the desired conclusion by using Proposition~\ref{subseqconv}.
\end{proof}

\begin{remark} {\rm Observe that if there exists $k_0 \in \mathbb{N}$ such that $\varepsilon_k \leq \delta\|d^k\|^2$ for all $k \geq k_0$ in Theorem \ref{teoS3}, then the estimate in \eqref{EQteoS3} implies that
\begin{equation*}
\phi(x^{k+1}) \leq \phi(x^k)\;\mbox{ whenever }\;k\geq k_0.
\end{equation*}
In this case, it is sufficient supposing that $\displaystyle\lim_{k\to\infty}\varepsilon_k = 0$ to guarantee that $\displaystyle\lim_{k\to\infty}\|d^k\|=0$.}
\end{remark}

\begin{theorem}
If the sequence $(\nu_k)_{k\in \mathbb{N}}$ is chosen according to strategy {\rm\textbf{(A1)}} with $\Delta_{min}>0$ and $\sum_{k=0}^{\infty}\varepsilon_k <\infty$, then every accumulation point of $(x_k)_{k\in \mathbb{N}}$ is critical for \eqref{dcproblem}.
\end{theorem}
\begin{proof}
Since $\Delta_{min}>0$ and $\sum_{k=0}^{\infty}\varepsilon_k <\infty$, Remark~\ref{re:equi} implies that $(\nu_k)_{k\in \mathbb{N}}$ satisfies \textbf{(A2)}. Then Theorem~\ref{teoS2} tells us that every accumulation point of $(x_k)_{k\in \mathbb{N}}$ is a critical point for \eqref{dcproblem}.
\end{proof}

The last of this section not only establishes the criticality of accumulation points but also verifies the nonincreasing property of the value sequence.

\begin{theorem}
Let $\varepsilon_k = 0$ for all $k \in \mathbb{N}$. If the sequence $(\nu_k)_{k\in\mathbb{N}}$ is chosen according to strategy {\rm\textbf{(A1)}}, the following assertions hold:
\begin{enumerate}
\item [\bf(i)] The sequence $(\phi(x^k) + \nu _k)_{k\in\mathbb{N}}$ is nonincreasing and convergent.
	
\item[\bf(ii)] If $\displaystyle{\lim_{k\to\infty}\nu _k} = 0$, then every accumulation point of $(x^k)_{k\in\mathbb{N}}$, if any, is a critical point of \eqref{dcproblem}.
\end{enumerate}
\end{theorem}
\begin{proof}

To justify (i), observe that \eqref{desDelta} in Remark~\ref{re:equi} reduces for $\varepsilon_k=0$ to
\begin{equation*}
0 \leq \Delta_{min}\left( \phi(x^k) - \phi(x^{k+1}) + {\nu_k} \right) \leq \left(\phi(x^k) + {\nu_k} \right) - \left(\phi(x^{k+1}) + \nu_{k+1}\right),
\end{equation*}
which implies that $\phi(x^k) + {\nu_k} \leq \phi(x^{k+1}) + \nu_{k+1}$ for all $k\in\mathbb{N}$, i.e., $(\phi(x^k) + \nu _k)_{k\in\mathbb{N}}$  is nonincreasing. It follows from \textbf{(H2)} that either $(\phi(x^k))_{k\in \mathbb{N}}$ or $(\nu _k)_{k\in\mathbb{N}}$ are lower bounded. Thus the sequence $(\phi(x^k) + \nu _k)_{k\in\mathbb{N}}$ is bounded and  nonincreasing being convergent therefore. 

Finally, we verify assertion (ii). It follows from (i) and $\lim_{k\to+\infty}{\nu_k} = 0$ that the sequence $(\phi(x^k))_{k\in \mathbb{N}}$ is convergent. On the other hand, inequality \eqref{ineqdescent1} with $\varepsilon_k = 0$ in Proposition~\ref{descent1} yields
\begin{equation*}
\left( \displaystyle\frac{\sigma}{2} - \theta \right)\|d^k\|^2 \leq \phi(x^k) - \phi(x^{k+1}) + {\nu_k}\;\mbox{ for all }\;k\in\mathbb N.
\end{equation*}
Passing to the limit as $N\to\infty$ in the last inequality leads us to $\displaystyle\lim_{k\to\infty}\|d^k\|^2 =0$ and thus completes the proof by employing Proposition~\ref{subseqconv}. 
\end{proof}

\section{Iteration-Complexity Analysis}\label{sec7}

This section presents some results of {\em iteration-complexity bounds} for ${(x^k)}_{k\in\mathbb{N}}$ generated by Algorithm~2. We consider the cases where the sequence ${(\nu _{k})}_{k\in \mathbb{N}}$ is choosing according to \textbf{(A2)} and \textbf{(A3)}. The theorems below are based on Proposition~\ref{descent1}, which particularly implies that
\begin{equation}\label{complexity}
\left( \displaystyle\frac{\sigma}{2} -\theta \right)\|d^k\|^2 \leq \phi(x^k) - \phi(x^{k+1})   + {\nu_k} + \varepsilon_k\;\mbox{ for all }\;k\in\mathbb N.
\end{equation}

\begin{theorem}
Suppose that the sequence $(\nu _{k})_{k\in \mathbb{N}}$ is chosen according to strategy {\rm\textbf{(A2)}} and that $\sum_{k=0}^{\infty}{\varepsilon_k} <\infty$. Then for each $N\in\mathbb{N}$ we have the estimate
\begin{equation*}
\min\big\{\| d^{k}\|\;\big|\;k=0,1,\ldots,N-1\big\}\leq {\frac{ \sqrt{\phi(x^{0})-\Bar\phi+\sum _{k=0}^{\infty}\nu_{k}+ \sum _{k=0}^{\infty}{\varepsilon _{k}}}} { \sqrt{ \frac{\sigma}{2} -\theta }}}\frac{1}{ \sqrt{N}}.
\end{equation*}
Consequently, for a given accuracy $\varepsilon>0$, if $$N\geq \left({\phi(x^{0})- \Bar\phi+\sum _{k=0}^{\infty}\nu _{k} +\sum _{k=0}^{\infty}\varepsilon _{k} }\right)/\left[\Big({\frac{\sigma}{2} - \theta}\Big) \varepsilon^{2}\right],$$ then  $\min \{\| d^{k}\|\;\big|\;k=0,1,\ldots,N-1\}\leq \varepsilon$.
\end{theorem}

\begin{proof}
By assumption \textbf{(H2)}, $\Bar \phi=\inf_{x\in\mathbb{R}^n}\phi(x) \leq \phi(x^k)$, $k\in\mathbb{N}$, it follows from \eqref{complexity} that
\begin{align*}
    \sum _{k=0}^{N-1}\|d^{k}\|^{2} &\leq \frac{1}{\frac{\sigma}{2}-\theta}\Big(\phi(x^{0})-\phi(x^{N})+\sum _{k=0}^{N-1}\nu _{k} + \sum _{k=0}^{N-1}\varepsilon _{k}\Big)\\ &\leq \frac{1}{\frac{\sigma}{2}-\theta} \Big(\phi(x^{0})-\Bar\phi+\sum _{k=0}^{\infty}\nu _{k} + \sum _{k=0}^{\infty}\varepsilon _{k}\Big).
\end{align*}
This gives us in turn that
\begin{equation*}
N\,\min\{\|d^k\|^2\;\big|\;k=0,1, \ldots N-1\} \leq \frac{\phi(x^{0})-\Bar\phi+\sum _{k=0}^{\infty}\nu _{k} + \sum _{k=0}^{\infty}\varepsilon _{k}}{\frac{\sigma}{2}-\theta},
\end{equation*}
which brings us to the inequality
\begin{equation}\label{complexity1}
\min \left\{\| d^{k}\|\;\big|\;k=0,1,\ldots,N-1\right\}\leq {\frac{ \sqrt{\phi(x^{0})-\Bar\phi\sum _{k=0}^{\infty}\nu_{k}+ \sum _{k=0}^{\infty}{\varepsilon _{k}}}} { \sqrt{ \frac{\sigma}{2}-\theta }}}\frac{1}{ \sqrt{N}}.
\end{equation}
Moreover, given any $\varepsilon > 0$, we deduce from
\begin{equation*}
N\geq \displaystyle\frac{{\phi(x^{0})- \Bar\phi+\sum _{k=0}^{\infty}\nu_{k} +\sum _{k=0}^{\infty}\varepsilon _{k} }}{\Big({\frac{\sigma}{2} - \theta}\Big) \varepsilon^{2}}
\end{equation*}  
the fulfillment of the estimate
\begin{equation*}
\phi(x^{0})- \Bar\phi+\sum _{k=0}^{\infty}\nu_{k} +\sum _{k=0}^{\infty}\varepsilon _{k} \leq N\Big({\frac{\sigma}{2} - \theta}\Big) \varepsilon^{2},
\end{equation*}
which being combined with \eqref{complexity1} provides
\begin{align*}
\min \left\{\| d^{k}\|:~k=0,1,\ldots,N-1\right\}&\leq {\frac{ \sqrt{\phi(x^{0})-\Bar\phi+\sum _{k=0}^{\infty}\nu_{k}+ \sum _{k=0}^{\infty}{\varepsilon _{k}}}} { \sqrt{ \frac{\sigma}{2}-\theta }}}\frac{1}{ \sqrt{N}} \notag\\
&\leq \displaystyle\frac{\sqrt{N\varepsilon^2 \Big(\frac{\sigma}{2} - \theta\Big)}}{\sqrt{\Big(\frac{\sigma}{2} - \theta\Big)}}\frac{1}{\sqrt{N}}\notag= \varepsilon
\end{align*}
and thus completes the proof of the theorem.
\end{proof}

\begin{theorem}
Suppose that  the sequence ${(\nu_{k})}_{k\in \mathbb{N}}$  is chosen according to strategy {\rm\textbf{(A3)}}, and let $0< \xi <1/2$. If $(\varepsilon _{k})_{k\in \mathbb{N}}$ is such that $\varepsilon_k \leq \xi\Big(\frac{\rho}{2} - \theta\Big)||d^k||^2$ for all $k \in \mathbb{N}$ and if $v _{k}\leq \xi \Big(\frac{\rho}{2} - \theta\Big)\|d^{k}\|^{2}$ for all $k\geq k_{0}$, then for each $N\in\mathbb{N}$ with $N> k_0$ we have
\begin{equation*}
\min \{\|d^{k}\| : k=0,1,\ldots,N-1\} \leq {\frac{ \sqrt{\phi(x^{0})-\Bar\phi+\sum _{k=0}^{k_0-1}\nu_{k}+\sum _{k=0}^{k_0-1}\varepsilon _{k}}}{ \sqrt{ (1-\xi)\Big(\frac{\sigma}{2} - \theta\Big) }}}\frac{1}{ \sqrt{N}}.
\end{equation*}
Consequently, given $\varepsilon>0$ and $k_{0}\in \mathbb{N}$ with $\nu_{k}\leq \xi \Big(\frac{\sigma}{2} - \theta\Big)\|d^{k}\|^{2}$ for $k\geq k_{0}$, it follows from
$$
N\geq \max \left\{k_0, \frac{ \phi(x^{0})- \Bar\phi+\sum _{k=0}^{k_0-1}\nu_{k} +\sum _{k=0}^{k_0-1}\varepsilon _{k}}{\Big(\frac{\sigma}{2} - \theta\Big)(1-\xi ) \varepsilon^{2}}\right\}
$$ 
that $\min \left\{\| d^{k}\|\;\big|\;~k=0,1,\ldots,N-1\right\}\leq \varepsilon$.
\end{theorem}

\proof Pick $\xi \in (0,1/2)$ and $k_0 \in \mathbb{N}$ with ${\nu_k} \leq \xi\Big(\frac{\sigma}{2} - \theta\Big)\|d^k\|^2$ whenever $k \geq k_0$. By \eqref{complexity}, we get $\Big(\frac{\sigma}{2} - \theta\Big)\|d^k\|^2 \leq \phi(x^k) - \phi(x^{k+1}) + {\nu_k} +\varepsilon_k$ as $k=0,1,\ldots, N-1$. Summing up the last inequality from $k=0$ to $K=N-1$ and using assumption \textbf{(H2)} give us
$$
\Big(\frac{\sigma}{2} - \theta\Big)\sum_{k=0}^{N-1}\|d^k\|^2 \leq \phi(x^0) - \bar{\phi}+ \sum_{k=0}^{k_0 - 1}\nu_k + \sum_{k=k_0}^{N - 1}\nu _k + \sum_{k=0}^{k_0 - 1}\varepsilon _k + \sum_{k=k_0}^{N - 1}\varepsilon _k.
$$
Due to $\varepsilon_k \leq \xi\Big(\frac{\sigma}{2} - \theta\Big)\|d^k\|^2$ and ${\nu_k} \leq \xi\Big(\frac{\sigma}{2} - \theta\Big)\|d^k\|^2$ as  $k \in \mathbb{N}$, the last inequality becomes
\begin{align*}
\sum_{k=0}^{N-1}\Big(\frac{\sigma}{2} - \theta\Big)\|d^k\|^2 &\leq \phi(x^0) - \bar{\phi}+ \sum_{k=0}^{k_0 - 1}\nu_k + \sum_{k=k_0}^{N - 1}\xi\Big(\frac{\sigma}{2} - \theta\Big)\|d^k\|^2 + \sum_{k=0}^{k_0 - 1}\varepsilon _k\\ &+ \sum_{k=k_0}^{N - 1}\xi\Big(\frac{\sigma}{2} - \theta\Big)\|d^k\|^2\notag\\
&= \phi(x^0) - \bar{\phi}+ \sum_{k=0}^{k_0 - 1}\nu_k +  \sum_{k=0}^{k_0 - 1}\varepsilon _k + 2\sum_{k=k_0}^{N - 1}\xi\Big(\frac{\sigma}{2} - \theta\Big)\|d^k\|^2, 
\end{align*}
which implies in turn that
\begin{equation*}
\sum_{k=0}^{N-1}\Big(\frac{\sigma}{2} - \theta\Big)\|d^k\|^2\leq \phi(x^0) - \bar{\phi}+ \sum_{k=0}^{k_0 - 1}\nu_k +  \sum_{k=0}^{k_0 - 1}\varepsilon _k + 2\sum_{k=k_0}^{N - 1}\xi\Big(\frac{\sigma}{2} - \theta\Big)\|d^k\|^2.
\end{equation*}
Therefore, we arrive at the estimate
\begin{equation*}
\sum_{k=0}^{N-1}(1-2\xi)\|d^k\|^2 \leq \phi(x^0) - \bar{\phi}+ \sum_{k=0}^{k_0 - 1}{\nu_k} + \sum_{k=0}^{k_0 - 1}\varepsilon _k
\end{equation*}
that brings us to the inequality 
\begin{equation*}
N\min\left\{ \|d^k\|^2\;\big|\;~ k= 0,1, \dots, N-1 \right\} \leq \displaystyle\frac{\phi(x^0) - \bar{\phi}+\sum_{k=0}^{k_0 - 1}{\nu_k} + \sum_{k=0}^{k_0 - 1}\varepsilon _k}{(1-2\xi)\Big(\frac{\sigma}{2} - \theta\Big)},
\end{equation*}
which readily ensures that 
\begin{equation}\label{complexity2}
\min\left\{ \|d^k\|\;\big|\; k= 0,1, \ldots, N-1 \right\} \leq \displaystyle\frac{\sqrt{\phi(x^0) - \bar{\phi}+\sum_{k=0}^{k_0 - 1}{\nu_k} + \sum_{k=0}^{k_0 - 1}\varepsilon _k}}{\sqrt{(1-2\xi)\Big(\frac{\sigma}{2} - \theta\Big)}}\frac{1}{\sqrt{N}}
\end{equation}
and thus verifies the first estimate claimed in the theorem. Furthermore, the condition
\begin{equation*}
\max \left\{ k_0, \displaystyle\frac{\phi(x^0) - \bar{\phi}+\sum_{k=0}^{k_0 - 1}{\nu_k} + \sum_{k=0}^{k_0 - 1}\varepsilon _k}{\Big(\frac{\sigma}{2} - \theta\Big)(1 - 2\xi)\varepsilon^2} \right\} \leq N, 
\end{equation*}
yields, in particular, the inequality
\begin{equation*}
{\phi(x^0) - \bar{\phi}+\sum_{k=0}^{k_0 - 1}{\nu_k} + \sum_{k=0}^{k_0 - 1}\varepsilon _k}  \leq N\Big(\frac{\sigma}{2} - \theta\Big)(1-2\xi)\varepsilon^2,
\end{equation*}
which being combined with \eqref{complexity2} justifies that
\begin{align*}
\min\left\{ \|d^k\|\;\big|\; k= 0,1, \ldots, N-1 \right\} &\leq \displaystyle\frac{\sqrt{\phi(x^0) - \bar{\phi}+\sum_{k=0}^{k_0 - 1}{\nu_k} + \sum_{k=0}^{k_0 - 1}\varepsilon _k}}{\sqrt{(1-2\xi)\Big(\frac{\sigma}{2} - \theta\Big)}}\frac{1}{\sqrt{N}}\notag\\
&\leq \frac{\sqrt{N\Big(\frac{\sigma}{2} - \theta\Big)(1-2\xi)\varepsilon^2}}{{\sqrt{(1-2\xi)\Big(\frac{\sigma}{2} - \theta\Big)N}}}\notag= \varepsilon
\end{align*}
and therefore completes the proof of the theorem.
\qed

\begin{theorem} Let the sequences of positive numbers ${(\nu_{k})}_{k\in \mathbb{N}}$ and  ${(\varepsilon _{k})}_{k\in \mathbb{N}}$ satisfy the conditions $\displaystyle \lim_{N \to \infty}\frac{\sum^{N-1}_{k=0}{\nu_k}}{N} = 0$ and $\displaystyle \lim_{N \to \infty}\frac{\sum^{N-1}_{k=0}\varepsilon_k}{N} = 0$, respectively. Then we have $\displaystyle\liminf_{k\to \infty}\|d^k\| = 0$.
\end{theorem}
\proof Taking the partial sum in \eqref{complexity} and using assumption \textbf{(H2)} give us the relationships
\begin{align*}
\sum_{k=0}^{N - 1}{\left( \displaystyle\frac{\rho}{2} -\theta \right)\|d^k\|^2} &\leq \sum_{k=0}^{N-1}{\left[\phi(x^k) - \phi(x^{k+1})\right]} + \sum_{k=0}^{N-1}{{\nu_k}} + \sum_{k=0}^{N-1}{\varepsilon_k}\\
&= \phi(x^0) - \phi(x^{N}) + \sum_{k=0}^{N-1}{{\nu_k}} + \sum_{k=0}^{N-1}{\varepsilon_k}\\
&\leq \phi(x^0) - \Bar\phi + \sum_{k=0}^{N-1}{{\nu_k}} + \sum_{k=0}^{N-1}{\varepsilon_k}.
\end{align*}
In this way, we arrive at the inequality
\begin{align*}
N\min\left\{ \|d^k\|^2 :~ k= 0, \dots, N-1 \right\} \leq \frac{1}{\frac{\rho}{2} -\theta}\left( \phi(x^0) - \Bar\phi + \sum_{k=0}^{N-1}{{\nu_k}} + \sum_{k=0}^{N-1}{\varepsilon_k} \right),
\end{align*}
which implies in turn the following estimates:
\begin{equation}\label{complexity3}
\min_{k= 0,\ldots, N-1} \|d^k\| \leq \sqrt{\frac{1}{\frac{\rho}{2} -\theta}\left( \frac{\phi(x^0) - \Bar{\phi}+\sum^{N-1}_{k=0}{\nu_k} + \sum^{N-1}_{k=0}\varepsilon_k}{N}\right)}.
\end{equation}
Passing to the limit as $N\to \infty$ in \eqref{complexity3}, we get that $\displaystyle\lim_{N \to \infty} \min\{\|d^k\|\;:\; k= 0, \ldots, N-1 \}$ $=0$. This means that there exists a subsequence of $(\|d^k\|)_{k\in\mathbb{N}}$ that converges to $0$ as $k \to \infty$. Recalling that $\|d^k\| > 0$ for all $k\in \mathbb{N}$, it follows that $\displaystyle\liminf_{k\to\infty}{\|d^k\|} = 0$, and thus we are done with the proof.
\qed

\section{Numerical Illustration}\label{sec8}

The objective of the numerical illustrations is to verify in practice that the inexact solutions found by BDCA and nmBDCA satisfy the conditions of Algorithm InmBDCA. At this stage, we are not focused on analyzing the computational performance of the method compared to others or addressing the efficiency of the approach used to solve the subproblems. Instead, we aim to demonstrate that by computing an inexact solution through running the BDCA and nmBDCA methods for a finite number of steps in each subproblem, we are effectively applying our inexact versions of these methods.

The numerical illustrations in this section are conducted using MATLAB software. The initial points are randomly chosen within the box $[-10,10] \times [-10,10]$. To solve the subproblems, we use the \texttt{fminsearch} toolbox with the inner stop rule \texttt{optimset('TolX',1e-7,'TolFun',1e-7)} (the justification for using \texttt{fminsearch} rather than more efficient methods is simplicity). The stopping criterion for the algorithm is $\|x^{k+1}-x^k\|<10^{-5}$. In the Example~\ref{InmBDCAex1} below, the constants in the design of Algorithm~InmBDCA are set as $\rho = 0.6$, $\beta = 0.1$, $\bar{\lambda} = 1$, and $\theta = 0.2$.

MATLAB solves Example~\ref{InmBDCAex1} inaccurately when computing the subproblem using the \texttt{fminsearch} toolbox. In follows, we verify computationally that the solution found by MATLAB satisfies inequalities \eqref{2inexactnmBDCA} and \eqref{3inexactnmBDCA} in Figures~\ref{desBDCA} and \ref{desnmBDCA}.

\begin{example}[{\cite[Example~3.3]{ARAGON2019}}]\label{InmBDCAex1} {\rm 
Let $\phi:\mathbb{R}^2\to\mathbb{R}$ be given by 
$\phi(x,y):=x^2+y^2+x+y-|x|-|y|$. We can obtain 
the following DC decomposition of the cost function: $\phi(x,y) 
= g(x,y)-h(x,y)$, where $g(x,y):=\frac{3}{2}
(x^2+y^2)+x+y$ and $h(x,y:)=\frac{1}{2}
(x^2+y^2)+|x|+|y|$. The minimum point of $\phi$ is $x_{opt}=(-1, -1)$, and the optimum value is $\phi_{opt} =-2$.}
\end{example}

\begin{figure}[ht]
\centering \includegraphics[width=0.5\linewidth]{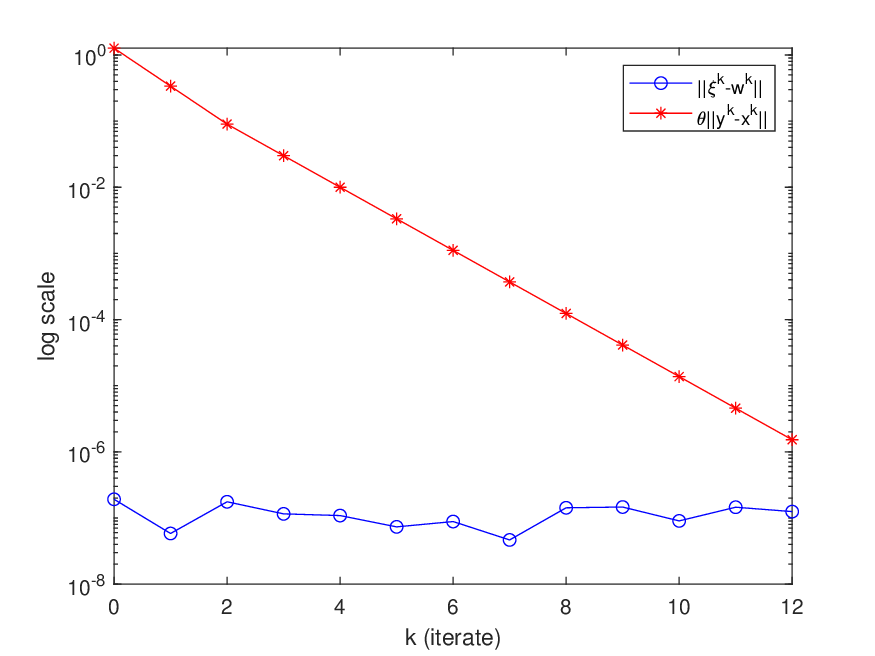}
\caption{Example \ref{InmBDCAex1} starting from $x^0 =(6.2945,8.1158)$.}
\label{desBDCA}
\end{figure}

In Example~\ref{InmBDCAex2}, the constants in the design of Algorithm~2 are set as $\rho = 0.6$, $\beta = 0.1$, $\bar{\lambda} = 1$, and $\theta = 0.2$. The sequence of parameters $ {(\nu_k)}_{k\in \mathbb{N}}$ are chosen as ${\nu_k} = 0.01\frac{\|d^k\|^2}{k+1}$ for all $k \in \mathbb{N}$.

\begin{example}[{\cite[Example 3.4]{ARAGON2019}}]\label{InmBDCAex2} {\rm 
Let $\phi:\mathbb{R}^2\to\mathbb{R}$ given by 
$\phi(x,y):=\frac{1}{2}(x^2+y^2)+|x|+|y|-\frac{5}{2}x$. We have the following DC decomposition of the cost function:
$\phi(x,y) = g(x,y)-h(x,y)$, where $g(x,y):=x^2+y^2+|x|+|y|-\frac{5}{2}x$ and $h(x,y):=\frac{1}{2}
(x^2+y^2)$. The minimum point of $\phi$ is $x_{opt}=(1.5, 0)$, and the optimum value is $\phi_{opt} =-1.125$.}
\end{example}

\begin{figure}[ht]
\centering
\includegraphics[width=0.5\linewidth]{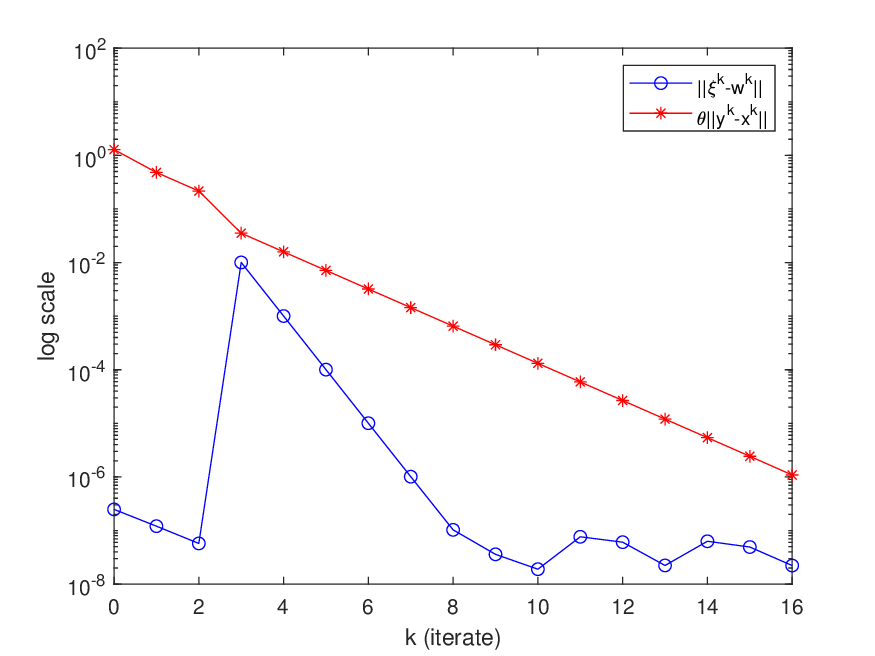}
\caption{Example \ref{InmBDCAex2} starting from $x^0=(-4.4615,-9.0766)$.} 
\label{desnmBDCA}
\end{figure}

\section{Conclusions}\label{sec9}

In this paper, we design a novel inexact algorithm to solve problems of DC programming, where both functions in DC compositions may be nondifferentiable. This algorithm can be viewed as a inexact version of Boosted Difference of Convex Algorithm (BDCA) in DC programming, while facing inexactness and nondifferentiability in the new algorithm requires a nonmonotone linesearch to maintain descent directions. Our major result establishes criticality of any accumulation point of the iterative sequence generated by the proposed algorithm. 

Directions of our future research include the following. First we intend to find conditions under which the criticality of accumulation points can be replaced by their stationarity in the senses discussed in Remark~\ref{stat}. Then our attention will be addressed to establishing global convergence of the entire iterative sequence to critical/stationary points of the algorithm. Deriving convergence rates is yet another goal of our future research. In this way, we plan to find conditions allowing us to include our algorithm into the abstract scheme of global convergence and convergent rates developed recently in \cite{bento24}, which covers the original exact version of BDCA introduced in  \cite{ARAGON2017}. Finally, we aim at considering problems of DC programming involving extended-real-valued functions, which makes it possible to deal with problems of constrained optimization.

\section*{Acknowledgements}
The research of the first author was partly supported by National Council for Scientific and Technological Development of Brazil (CNPq) Grant 304666/2021-1. The research of the second author was partly supported by the US National Science Foundation under grant DMS-1808978, by the Australian Research Council Discovery Project DP-190100555, and by Project 111 of China under grant D21024. The research of the third author was partly supported by Brazilian Federal Agency for Support and Evaluation of Graduate Education (CAPES). The research of the fourth author was partly supported by CNPq Grant 315937/2023-8.


\end{document}